\documentclass[11pt]{amsart}

\usepackage{amsmath,amsfonts,amsthm,graphicx,fullpage}
\usepackage[foot]{amsaddr}
\usepackage[hidelinks]{hyperref}

\DeclareMathOperator{\HJ}{HJ}
\DeclareMathOperator{\Bin}{Bin}

\allowdisplaybreaks 

\newtheorem{theorem}{Theorem}

\newtheorem{lemma}{Lemma}[section]

\title{An upper bound for the Hales--Jewett number $\HJ(4,2)$}

\author{Mikhail Lavrov}
\address{Department of Mathematical Sciences, Carnegie Mellon University, Pittsburgh, PA 15213}
\email{mlavrov@andrew.cmu.edu}

\begin{document}

\begin{abstract}
	We show that for $n$ at least $10^{11}$, any 2-coloring of the $n$-dimensional grid $[4]^n$ contains a monochromatic combinatorial line. This is a special case of the Hales--Jewett Theorem~\cite{hales63}, to which the best known general upper bound is due to Shelah~\cite{shelah88}; Shelah's recursion gives an upper bound between $2 \uparrow \uparrow 7$ and $2 \uparrow \uparrow 8$ for the case we consider, and no better value was previously known.
\end{abstract}

\maketitle

\section{Introduction}

Consider $r$-colorings of the $n$-dimensional grid $[t]^n = \{1, 2, \dots, t\}^n$. We define a \emph{combinatorial line} in $[t]^n$ to be an injective function $\ell : [t] \to [t]^n$ such that for each coordinate $1 \le i \le n$, $\ell_i$ is either constant or the identity function on $[t]$. An example of such a line in $[4]^5$ is the function
\[
	\ell(x) = (3, x, 1, x, 4)
\]
whose image is the set of four points $\{(3,1,1,1,4), (3, 2, 1, 2, 4), (3, 3, 1, 3, 4), (3, 4, 1, 4, 4)\}$. A classic result in Ramsey theory, the Hales--Jewett Theorem~\cite{hales63}, asserts that for all values of the parameters $t$ and $r$, there exists a sufficiently large $n$ such that any $r$-coloring of $[t]^n$ will contain a monochromatic combinatorial line (that is, a line such that the points in its image are all assigned the same color). The Hales--Jewett number $\HJ(t,r)$ is defined to be the least $n$ which suffices.

A pigeonhole argument is enough to show that $\HJ(2,r) = r$ for all $r$. Moreover, Hindman and Tressler~\cite{hindman14} have shown that $\HJ(3,2) = 4$. For more difficult cases, no exact values are known, and the best upper bounds were shown by Shelah~\cite{shelah88}. For $\HJ(4,2)$, this upper bound is already enormous. Shelah proves that if $\HJ(t-1,r) \le n$, then $\HJ(t,r) \le n f(n, r^{t^n})$, where $f(\ell, k)$ is an upper bound for a related problem that satisfies $k \uparrow \uparrow \ell \le f(\ell,k) \le k \uparrow \uparrow (2\ell)$. Starting from $\HJ(3,2)=4$, we get a bound for $\HJ(4,2)$ between $2 \uparrow \uparrow 7$ and $2 \uparrow \uparrow 8$.

On the other hand, the best lower bounds for the Hales--Jewett numbers, which can be obtained from the van der Waerden theorem (as in~\cite{shelah88}), are very far away from these upper bounds. For instance, the integers $\{1, 2, \dots, 34\}$ can be 2-colored in such a way that no 4-term arithmetic progression is monochromatic~\cite{chvatal70}. This coloring can be used to define a coloring of $[4]^{11}$ with no monochromatic combinatorial line, by coloring a point $(x_1, x_2, \dots, x_{11})$ with the color of $x_1 + x_2 + \cdots + x_{11} - 10$, which shows that $\HJ(4,2) > 11$. In general, this argument yields merely exponential lower bounds on the Hales--Jewett numbers: Berlekamp~\cite{berlekamp68} showed that for prime $p$, we can 2-color $p \cdot 2^p$ consecutive integers with no monochromatic $p$-term arithmetic progression. This opens up the possibility that the true values of the Hales--Jewett numbers may be much smaller.

We extend the boundary of which Hales--Jewett numbers are known to have reasonable values by proving the following result:

\begin{theorem}
	\label{main-theorem}
	Whenever the $10^{11}$-dimensional grid $[4]^{10^{11}}$ is 2-colored, there exists a monochromatic combinatorial line. That is, $\HJ(4, 2) \le 10^{11}$.
\end{theorem}

\section{Setup}

Given a combinatorial line $\ell : [4] \to [4]^n$, we define its \emph{length} $|\ell|$ to be the number of coordinates $1 \le i \le n$ for which $\ell_i(x)$ varies with $x$. For example, the line given by $\ell(x) = (3,x,1,x,4)$ has length $2$. (To justify this terminology, note that $|\ell|$ is the Hamming distance between the two endpoints $\ell(1)$ and $\ell(4)$, or indeed between any two points of $\ell$.)

Fix a $2$-coloring of $[4]^n$. For each length $k$, we classify the combinatorial lines of length $k$ into three types, and count their densities:
\begin{itemize}
\item $p_2(k)$ is the fraction of lines of length $k$ which have 2 points of each color.
\item $p_3(k)$ is the fraction of lines of length $k$ which have 3 points of one color, and 1 of the other.
\item $p_4(k)$ is the fraction of monochromatic lines of length $k$.
\end{itemize}
We are also interested in the fraction of pairs of collinear points (on lines of length $k$) assigned the same color. On each line, there are 6 pairs of points. For lines counted by $p_2$, 2 pairs are monochromatic; for lines counted by $p_3$, 3 pairs are monochromatic; for lines counted by $p_4$, all 6 pairs are monochromatic. Therefore
\begin{equation}
	\label{eq:density}
	q(k) := \frac13 p_2(k) + \frac12 p_3(k) + p_4(k)
\end{equation}
counts the fraction of monochromatic pairs of points on lines of length $k$.

The grid $[4]^n$ contains large ``cliques'': sets of points in which any two are collinear. In such a clique, the number of monochromatic pairs is least when the colors are balanced, and even then is close to $\frac12$. Thus, we expect that for at least some lengths $k$, $q(k) > \frac12 - \epsilon$ for some $\epsilon$ that goes to $0$ with $n$. This intuition is correct in a way that we will make more precise.

If we could show the stronger statement that $q(k) > \frac12$ for some $k$, the proof would be complete: $p_4(k)$ occurs in \eqref{eq:density} with a coefficient greater than $\frac12$, so we would know that $p_4(k) > 0$, which means that a monochromatic line exists. 

Even if $q(k) < \frac12$, a large value of $q(k)$ gives partial information: either $p_4(k) > 0$, or else $p_3(k)$ is close to $1$. Therefore we can also prove that $p_4(k) > 0$ by showing that for some $k$, $q(k)$ is close to $\frac12$, but $p_3(k)$ is bounded away from $1$. More formally, we can solve \eqref{eq:density} for $p_4(k)$ (substituting $p_2(k) = 1 - p_3(k) - p_4(k)$) to get
\begin{equation}
	\label{eq:ultimate}
	p_4(k) = \frac32\left(q(k) - \frac16 p_3(k) - \frac13\right).
\end{equation}
We will prove that a monochromatic line exists by showing that the right-hand side of \eqref{eq:ultimate} is positive for some $k$.

\section{Showing that $q(k)$ is close to $\frac12$}

It is hopeless to show that $q(k)$ approaches $\frac12$ for any individual $k$. For example, the ``checkerboard'' coloring, which colors a point by the sum of its coordinates modulo 2, has $p_2(k) = 1$, and therefore $q(k) = \frac13$, for all odd $k$.  Instead, we prove an inequality for a weighted sum of the first $\kappa$ values of $q(k)$, where $\kappa$ is a parameter to be determined later.

\subsection{A bound for $n$-dimensional hypercubes}

We begin by considering collinear pairs in the hypercube $[2]^n$. Here, lines consist of only 2 points, and therefore $q(k)$, defined as before, simply counts monochromatic lines of length $k$.

\begin{lemma}
\label{lemma:hypercube}
For every $\kappa \le n$, whenever $[2]^n$ is $2$-colored,
\begin{equation}
	\label{eq:hypercube}
	\sum_{k=1}^\kappa (\kappa-k+1) q(k) \ge \frac{\kappa^2-1}{4} \left(1 - \kappa \sqrt{\frac{2}{\pi n}}\right).
\end{equation}
\end{lemma}
\begin{proof}
The hypercube $[2]^n$ is the union of $n!$ chains of length $n+1$: maximal sequences of points $C_0, \dots, C_n$ such that any two points $C_i$ and $C_j$ are collinear. For each permutation $\sigma$ of $[n]$, we obtain such a chain by letting $C_i(\sigma)$ be the point with $2$ in the coordinates $\sigma(1), \sigma(2), \dots, \sigma(i)$ and $1$ in all others. The line through $C_i(\sigma)$ and $C_j(\sigma)$, for $i<j$, has length $j-i$. Any permutation $\sigma'$ such that $\{\sigma(1), \dots, \sigma(i)\} = \{\sigma'(1), \dots, \sigma'(i)\}$ and $\{\sigma(i+1), \dots, \sigma(j)\} = \{\sigma'(i+1), \dots, \sigma'(j)\}$ will satisfy $C_i(\sigma) = C_i(\sigma')$ and $C_j(\sigma) = C_j(\sigma')$; therefore $C_i(\sigma)$ and $C_j(\sigma)$ occur together in $i! (j-i)! (n-j)!$ chains.

Fix any 2-coloring of $[2]^n$. Let $Q_{i,j}(\sigma) = 1$ if $C_i(\sigma)$ and $C_j(\sigma)$ are given the same color, and 0 otherwise. Then the total number of monochromatic points on all ${n \choose k}2^{n-k}$ lines of length $k$ is given by
\begin{align*}
{n \choose k} 2^{n-k} q(k) &= \sum_{\sigma} \sum_{i=0}^{n-k} \frac{Q_{i,i+k}(\sigma)}{i! k! (n-i-k)!} \\
&= \frac{1}{n!}  {n \choose k} \sum_{\sigma}\sum_{i=0}^{n-k} {n-k \choose i} Q_{i,i+k}(\sigma).
\end{align*}
Therefore
\begin{equation}
	\label{eq:chain-average}
q(k) = \frac1{n!} \sum_{\sigma} \left(\sum_{i=0}^{n-k} \frac{{n-k \choose i}}{2^{n-k}} Q_{i,i+k}(\sigma)\right).
\end{equation}
Let
\[
	q(k,\sigma) = \sum_{i=0}^{n-k} \frac{{n-k \choose i}}{2^{n-k}} Q_{i,i+k}(\sigma).
\]
Then equation~\eqref{eq:chain-average} shows that $q(k)$ is the average of $q(k, \sigma)$ over all $\sigma$. To complete the proof of Lemma~\ref{lemma:hypercube}, it suffices to show that for each permutation $\sigma$, inequality~\eqref{eq:hypercube} holds with $q(k, \sigma)$ in place of $q(k)$. 

Define $w_{i,j} :=\frac{{n - (j-i) \choose i}}{2^{n-(j-i)}}$, a shorthand for the coefficient of $Q_{i,j}(\sigma)$ in $q(j-i, \sigma)$. We have
\begin{equation}
	\label{eq:sum-of-cliques}
	\sum_{k=1}^\kappa (\kappa-k+1) q(k, \sigma) \ge \sum_{h=0}^{n-\kappa} \left(\sum_{h \le i < j \le h+k} w_{i,j} Q_{i,j}(\sigma)\right)
\end{equation}
because each term $w_{i,j} Q_{i,j}(\sigma)$ occurs in the right-hand side of \eqref{eq:sum-of-cliques} for at most $\kappa - (j-i) + 1$ values of $i$; fewer if $j < \kappa $ or if $i > n- \kappa$. 

Each sum
\[
	\sum_{h \le i<j \le h+k} w_{i,j} Q_{i,j}(\sigma)
\]
counts (with varying weights) the number of monochromatic pairs among the $\kappa+1$ points $C_h(\sigma)$, $C_{h+1}(\sigma)$, \dots, $C_{h+\kappa}(\sigma)$, any two of which are collinear. There must be at least $2{(\kappa+1)/2 \choose 2} = \frac{\kappa^2-1}{4}$ such pairs; their number is minimized if half the points receive one color and half receive the other. We do not know which weights correspond to those pairs. However, at the very least, we have the lower bound
\[
	\sum_{h \le i < j \le h+k} w_{i, j} Q_{i, j}(\sigma) \ge \frac{\kappa^2-1}{4}  w_h^*,
\]
where $w_h^*$ is the least of all weights $w_{i, j}$ for $h \le i \le j \le h+\kappa$. (We allow $i=j$ to simplify calculations later, though such a weight does not occur in the sum.) Substituting this lower bound into inequality~\eqref{eq:sum-of-cliques}, we get
\[
	\sum_{k=1}^\kappa (\kappa-k+1) q(k, \sigma) \ge \frac{\kappa^2-1}{4}  \sum_{h=0}^{n-\kappa} w_h^*.
\]
It remains to find a lower bound for the sum of the $w_h^*$.

From Pascal's identity it follows that $w_{i,j} = \frac12 (w_{i-1,j} + w_{i,j+1})$. Therefore each coefficient $w_{i,j}$ is a weighted average of some of the coefficients
\[
	w_{h,h},  w_{h,h+1}, \dots, w_{h,h+\kappa}, w_{h+1,h+\kappa}, \dots, w_{h+\kappa,h+\kappa}.
\]
Since all of these coefficients are included in the minimum defining $w_h^*$, we know that $w_h^*$ must be one of these. Furthermore, this sequence is unimodal, so $w_h^* = \min\{w_{h,h}, w_{h+\kappa,h+\kappa}\}$. 

The sequence $w_{0,0}, w_{1,1}, \dots, w_{n,n}$ is just the sequence $\frac{{n \choose 0}}{2^n}, \frac{{n \choose 1}}{2^n}, \dots, \frac{{n \choose n}}{2^n}$, and is also unimodal. So the sum $\sum_{h=0}^{n-\kappa} w_h^*$ will begin by summing $w_{h,h}$ and eventually switch to summing $w_{h+\kappa,h+\kappa}$, skipping some $\kappa$ terms. Therefore
\[
	\sum_{h=0}^{n-\kappa} w_h^* \ge \sum_{h=0}^n w_{h,h} - \kappa \max_{0\le h \le n} w_{h,h} = 1 - \kappa \frac{{n \choose \lfloor n/2 \rfloor}}{2^n} \ge 1 - \kappa \sqrt{\frac{2}{\pi n}}.
\]
It follows that
\[
	\sum_{k=1}^\kappa (\kappa-k+1) q(k, \sigma) \ge \frac{\kappa^2-1}{4}  \left(1 - \kappa \sqrt{\frac{2}{\pi n}}\right)
\]
and by averaging this inequality over all permutations $\sigma$ and applying equation~\eqref{eq:sum-of-cliques}, we obtain the desired inequality~\eqref{eq:hypercube}.
\end{proof}

\subsection{Extending the bound to the grid $[4]^n$}

By giving away another error term, we can extend Lemma~\ref{lemma:hypercube} to all collinear pairs in $[4]^n$.

\begin{lemma}
	\label{lemma:collinear-pairs}
For every $\kappa \le \frac{n}{4}$, whenever $[4]^n$ is 2-colored,
\[
	\sum_{k=1}^\kappa (\kappa-k+1) q(k) \ge \frac{\kappa^2-1}{4} \left(1 - e^{-(n-\kappa)/8} - 3 \kappa \sqrt{ \frac{2}{\pi(n-\kappa)}}\right).
\]
\end{lemma}
\begin{proof}
We say that a collinear pair of points $\ell(a), \ell(b)$ for some line $\ell$ and some $a,b \in [4]$ has \emph{type~$m$} if there are $m$ coordinates in total in which either point is equal to $a$ or $b$; in other words, $\ell_i(x)$ is the constant $a$ or $b$ for $m - |\ell|$ values of $i$. We define $q(k, m)$ to be the fraction of collinear pairs of type~$m$ and on lines of length~$k$ which are monochromatic.

The type of a collinear pair matters because a collinear pair of type~$m$ is contained in the $m$-dimensional subcube of $[4]^n$ obtained by letting all $m$ coordinates of either point which are equal to $a$ or $b$ vary freely between the two values. In this $m$-dimensional subcube, two points are collinear if and only if the corresponding points of $\{a, b\}^m$ (obtained by dropping all coordinates not equal to $a$ or $b$) are collinear, so it has the structure of the hypercube $[2]^m$. The fraction of collinear pairs in this subcube which are monochromatic satisfies Lemma~\ref{lemma:hypercube}. By averaging over all $m$-dimensional subcubes, which cover each collinear pair of type~$m$ exactly once, we obtain
\begin{equation}
	\label{eq:type-m}
	\sum_{k=1}^\kappa (\kappa-k+1) q(k,m) \ge \frac{\kappa^2-1}{4} \left(1 - \kappa \sqrt{\frac{2}{\pi m}}\right).
\end{equation}

There are $6{n \choose k}4^{n-k}$ collinear pairs on lines of length~$k$; of them, ${m \choose k}2^{m-k}$ are in each $m$-dimensional subcube, and there are $6{n \choose m}2^{n-m}$ such subcubes. So the fraction of lines of length~$k$ which have type $m$ is
\[
	\frac{{m \choose k}2^{m-k} {n \choose m}2^{n-m}}{{n \choose k} 4^{n-k}} = \frac{{n-k \choose m-k}}{2^{n-k}}.
\]
Therefore we may express $q(k)$ as a weighted average of all the $q(k,m)$ by
\begin{equation}
	\label{eq:m-average}
	q(k) = \sum_{m=k}^n \frac{{n-k \choose m-k}}{2^{n-k}} q(k,m).
\end{equation}

Unfortunately, the weight of $q(k,m)$ in this average depends on $k$ as well as $m$, which prevents us from simply averaging inequality~\eqref{eq:type-m} over all $m$. To fix this problem, we replace the weights in~\eqref{eq:m-average} by lower bounds independent of $k$, which will result in an inequality relating $q(k)$ to $q(k,m)$. (We will assume that $1 \le k \le \kappa \le \frac n4$.)

For $m \le \frac{n}{4}$, our lower bound will be $0$: we drop all terms where $m$ is too low, because the statement of inequality~\eqref{eq:type-m} is too weak in such cases. Otherwise, we want to replace the weight by the minimum of ${n-k \choose m-k}2^{-(n-k)}$ over all $k \le \kappa$.

From Pascal's identity, we have ${n \choose r}2^{-n} = \frac12\left({n-1 \choose r-1} 2^{-(n-1)} + {n-1 \choose r} 2^{-(n-1)}\right)$. Applying this iteratively, we can express each ${n-k \choose m-k}2^{-(n-k)}$ as a weighted average of some of 
\[
	\frac{{n-\kappa \choose m-\kappa}}{2^{n-\kappa}}, \frac{{n-\kappa \choose m-\kappa+1}}{2^{n-\kappa}}, \dots, \frac{{n-\kappa \choose m}}{2^{n-\kappa}}.
\]
This sequence is unimodal, so the minimum is achieved at one of the endpoints, and we may replace equation~\eqref{eq:m-average} by
\begin{equation}
	\label{eq:m-lower-bound}
	q(k) \ge \sum_{m=n/4}^n \frac{\min\left\{ {n-\kappa \choose m-\kappa}, {n-\kappa \choose m}\right\}}{2^{n-\kappa}} q(k,m).
\end{equation}

Sum the inequality~\eqref{eq:type-m} over all $m \ge \frac{n}{4}$ with weights as in inequality~\eqref{eq:m-lower-bound}. The right-hand side of~\eqref{eq:m-lower-bound} will be smallest when $m = \frac{n}{4}$, so we may use that value for all $m$. We obtain
\begin{equation}
	\label{eq:nearly-there}
	\sum_{k=1}^\kappa (\kappa-k+1) q(k) \ge \frac{\kappa^2-1}{4} \left(1 - \kappa\sqrt{\frac{2}{\pi n/4}}\right) \sum_{m=n/4}^n \frac{\min\left\{ {n-\kappa \choose m-\kappa}, {n-\kappa \choose m}\right\}}{2^{n-\kappa}}.
\end{equation}
It remains to simplify the right-hand side. 

The omission of the first $n/4$ terms of the sum in~\eqref{eq:nearly-there} results in an error of $\sum_{m<n/4} {n-\kappa \choose m-\kappa} 2^{-(n-\kappa)}$, which is simply the binomial probability $\Pr[\Bin(n-\kappa, \frac12) < \frac n4 - \kappa]$. By the Chernoff bound (see, e.g., \cite{alon08}),
\[
	\Pr\left[\Bin(n-\kappa, \tfrac12) < \frac n4 - \kappa\right] < \Pr\left[\Bin(n-\kappa, \tfrac12) < \frac{n-\kappa}{4}\right] \le \exp \left(-\frac{n-\kappa}{8}\right).
\]

With these initial terms, the sum  in~\eqref{eq:nearly-there} would be equal to 1, except for skipping $\kappa$ terms near the middle, which occurs when the minimum switches from selecting ${n-\kappa \choose m-\kappa}$ to selecting ${n-\kappa \choose m}$. Each of these terms is at most ${n-\kappa \choose (n-\kappa)/2} 2^{-(n-\kappa)} \le \sqrt{\frac{2}{\pi (n-\kappa)}}$, so we lose at most $\kappa$ times this quantity. Therefore the sum in inequality~\eqref{eq:nearly-there} satisfies
\[
	\sum_{m=n/4}^n \frac{\min\left\{ {n-\kappa \choose m-\kappa}, {n-\kappa \choose m}\right\}}{2^{n-\kappa}} \ge 1 - e^{-(n-\kappa)/8} - \kappa \sqrt{ \frac{2}{\pi(n-\kappa)}}.
\]
Combining the two error terms, we complete the proof.
\end{proof}

\section{Showing that $p_3(k)$ cannot be arbitrarily close to 1}

In this section, we say that a combinatorial line in a 2-colored grid $[4]^n$ is \emph{odd} if it has an odd number of points of each color. That is, an odd line has 3 points of one color and 1 point of the other, so it is exactly the type of line counted by $p_3(k)$.

To bound $p_3(k)$ away from 1, we first find a set of lines in $[4]^4$ which cannot all be odd:

\begin{lemma}
\label{lemma:fifteen}
Whenever $[4]^4$ is 2-colored, the $15$ lines
\begin{align*}
	\ell^1(x) &= (x, 2, 3, 4) & \ell^6(x) &= (x, 2, x, 4)  & \ell^{11}(x) &= (x, x, x, 4) \\
	\ell^2(x) &= (1, x, 3, 4) & \ell^7(x) &= (x, 2, 3, x) & \ell^{12}(x) &= (x, x, 3, x) \\
	\ell^3(x) &= (1, 2, x, 4) & \ell^8(x) &= (1, x, x, 4) & \ell^{13}(x) &= (x, 2, x, x) \\
	\ell^4(x) &= (1, 2, 3, x) & \ell^9(x) &= (1, x, 3, x) & \ell^{14}(x) &= (1, x, x, x) \\
	\ell^5(x) &= (x, x, 3, 4) & \ell^{10}(x) &= (1, 2, x, x) & \ell^{15}(x) &= (x,x,x,x)
\end{align*}
cannot all be odd.
\end{lemma}
\begin{proof}
A key observation is that each point of $[4]^4$ lies on an even number of these lines. The point $(1,2,3,4)$ lies on the 4 lines of length $1$, and no other. Take any other point $(x_1,x_2,x_3,x_4)$ expressible as $\ell^j(x)$ for some index $j$ and some $x \in \{1,2,3,4\}$. Any coordinate $i$ where $x_i \ne  i$ must be a variable coordinate of $\ell^j$; any coordinate $i$ where $x_i = i \ne x$ must be a constant coordinate of $\ell^j$. There is always exactly one coordinate where $x_i = i = x$, so there are 2 choices for $j$, depending on whether that coordinate is variable or constant.

If $[4]^4$ is $2$-colored, choose either of the colors, and add up the number of points of that color on each of the fifteen lines. This total must be even, because each point is counted an even number of times. However, 15 odd numbers cannot add up to an even total, so one of the lines must contribute an even number. Therefore not all 15 lines can be odd.
\end{proof}

Structures isomorphic to the set of lines $\ell^1, \dots, \ell^{15}$ occur many times in $[4]^n$, and in each such structure at most $\frac{14}{15}$ of the lines are odd. So our next step is to show that by (more or less) averaging over all such structures, we get an upper bound of $\frac{14}{15}$ for the overall densities of odd lines of certain lengths, up to an error term.

\begin{lemma}
\label{lemma:k-embedding}
For every $k \le \frac n4$,  whenever $[4]^n$ is 2-colored,
\begin{equation}
	\label{eq:embedding-average}
	\left(1 - \frac{16k^2}{n}\right)\left(\frac{4}{15} p_3(k) + \frac{6}{15} p_3(2k) + \frac{4}{15} p_3(3k) + \frac1{15} p_3(4k)\right) \le \frac{14}{15}.
\end{equation}
\end{lemma}
\begin{proof}
Fix a 2-coloring of $[4]^n$ and some $k \le \frac n4$. Let a \emph{$k$-embedding} of $[4]^4$ into $[4]^n$ be a function $L : [4]^4 \to [4]^n$ such that for each coordinate $1 \le i \le n$, $L_i$ is either constant or given by $L_i(x_1, x_2, x_3, x_4) = x_j$ for some $j \in \{1,2,3,4\}$. Moreover, we require that for each $j$, there are exactly $k$ coordinates in which $L_i$ varies with $x_j$. Let $\mathcal L_k$ be the set of all $k$-embeddings $[4]^4 \to [4]^n$.

Each $L \in \mathcal L_k$ induces a $2$-coloring of $[4]^4$, by taking the preimage under $L$ of the coloring of $[4]^n$. Moreover, a line $\ell : [4] \to [4]^4$ corresponds to a line $L \circ \ell : [4] \to [4]^n$, with $|L \circ \ell| = k |\ell|$, which is odd if and only if $\ell$ is odd in the induced coloring.

Count the number of odd lines $L \circ \ell^j$, where $L \in \mathcal L_k$ and $\ell^j$ is one of the 15 lines of Lemma~\ref{lemma:fifteen}. For a fixed $L$, at most 14 of the lines $L \circ \ell^j$ are odd; therefore we count at most $14|\mathcal L_k|$ odd lines total. 

Let $P_3(k)$ be the number of odd lines of length $k$ in $[4]^n$, related to the density $p_3(k)$ by $P_3(k) = {n \choose k} 4^{n-k} p_3(k)$. If each line of length $k$ could be expressed $M(k)$ times as $L \circ \ell^j$, we would have the inequality
\begin{equation}
	\label{eq:odd-line-symbolic}
	M(k) P_3(k) + M(2k) P_3(2k) + M(3k) P_3(3k) + M(4k) P_3(4k) \le 14 |\mathcal L_k|.
\end{equation}
Unfortunately, the number of ways to express a line $\ell : [4] \to [4]^n$ as $L \circ \ell^j$ depends on $\ell$; specifically, on the number of coordinates of $\ell$ with each constant value. Inequality~\eqref{eq:odd-line-symbolic} still holds, however, if we instead define $M(k)$ to be the minimum multiplicity of any line of length $k$.

We compute the minimum multiplicity for each of the four possible lengths. Below, let $n_1$, $n_2$, $n_3$, and $n_4$ denote the number of coordinates of $\ell$ with constant value 1, 2, 3, and 4, respectively.
\begin{itemize}
\item If $|\ell|=k$, then $\ell$ can be expressed as $L(x, 2, 3, 4)$ in ${n_2 \choose k} {n_3 \choose k} {n_4 \choose k}$ ways; we also get the corresponding counts for expressions of the form $L(1, x, 3, 4)$, $L(1, 2, x, 4)$, and $L(1,2,3,x)$. In total the line is counted with multiplicity ${n_2 \choose k} {n_3 \choose k} {n_4 \choose k} + {n_1 \choose k} {n_3 \choose k} {n_4 \choose k} + {n_1 \choose k} {n_2 \choose k} {n_4 \choose k} + {n_1 \choose k} {n_2 \choose k} {n_3 \choose k}$.  This sum is minimized when $n_1, n_2, n_3, n_4$ are as equal as possible, so 
\[
	M(k) \ge 4 {\frac{n-k}{4} \choose k}^3.
\]

\item If $|\ell|=2k$, then $\ell$ can be expressed as $L(x, x, 3, 4)$ in ${2k \choose k} {n_3 \choose k}{n_4 \choose k}$ ways; we also get the corresponding counts for expressions of the form $L(x, 2, x, 4)$ and so on, for a total of ${2k \choose k} \left({n_3 \choose k} {n_4 \choose k} + {n_2 \choose k} {n_4 \choose k} + {n_1 \choose k}{n_4 \choose k} + {n_2 \choose k} {n_3 \choose k} + {n_1 \choose k} {n_3 \choose k} + {n_1 \choose k} {n_2 \choose k}\right)$. This is, once again, minimized when $n_1, n_2, n_3, n_4$ are as equal as possible, so 
\[
	M(2k) \ge 6{2k \choose k} {\frac{n-2k}{4} \choose k}^2.
\]

\item If $|\ell|=3k$, then $\ell$ can be expressed as $L(1,x,x,x)$ or $L(x,2,x,x)$ or $L(x,x,3,x)$ or $L(x,x,x,4)$ in a total of 
${3k \choose k,k,k} \left({n_1 \choose k} + {n_2 \choose k} + {n_3 \choose k} + {n_4 \choose k}\right)$ ways, so 
\[
	M(3k) \ge 4 {3k \choose k,k,k} {\frac{n-3k}{4} \choose k}.
\]

\item Finally, if $|\ell|=4k$, then $\ell$ can only be expressed as $L(x,x,x,x)$, which can be done in 
\[
	M(4k) = {4k \choose k,k,k,k}
\]
ways.
\end{itemize}
We can further replace $|\mathcal L_k|$ by ${n \choose k,k,k,k,n-4k} 4^{n-4k}$. This allows us to rewrite inequality~\eqref{eq:odd-line-symbolic} as
\begin{multline*}
	4 {\frac{n-k}{4} \choose k}^3 {n \choose k} 4^{n-k} p_3(k) + 6{2k \choose k} {\frac{n-2k}{4} \choose k}^2 {n \choose 2k} 4^{n-2k} p_3(2k)+ \\
	+ 4 {3k \choose k,k,k} {\frac{n-3k}{4} \choose k} {n \choose 3k} 4^{n-3k} p_3(3k) + {4k \choose k,k,k,k} {n \choose 4k} 4^{n-4k} p_4(4k) \le \\ \le {n \choose k,k,k,k,n-4k} 4^{n-4k}.
\end{multline*}
This inequality can be simplified by factoring out $\frac{4^{n-4k}}{k!^4}$ from each term. If we also replace falling powers $r(r-1)(r-2)(\cdots)(r-s+1)$ by $r^s$ on the right-hand side (as an upper bound) and by $(r-s)^s$ on the left-hand side (as a lower bound), we obtain
\begin{multline*}
	4(n-k)^k (n-5k)^{3k} p_3(k) + 6 (n-2k)^{2k} (n-6k)^{2k} p_3(2k) + \\ + 4(n-3k)^{3k} (n-7k)^k p_3(3k) + (n-4k)^{4k} p_3(4k) \le 14 n^{4k}.
\end{multline*}
Finally, dividing through by $n^{4k}$ yields factors such as $\left(1 - \frac{k}{n}\right)^k$. By iteratively applying the inequality $(1-u)(1-v) \ge 1-u-v$ for $u,v \ge 0$, we bound the first such factor:
\[
	\left(1 - \frac kn\right)^k \left(1 - \frac{5k}{n}\right)^{3k} \ge \left(1 - \frac{k^2}{n}\right) \left(1 - \frac{15k^2}{n}\right) \ge 1 - \frac{16k^2}{n}.
\]
Similarly, the factors of $\left(1 - \frac{2k}{n}\right)^{2k} \left(1 - \frac{6k}{n}\right)^{2k}$, $\left(1-\frac{3k}{n}\right)^{3k} \left(1 - \frac{7k}{n}\right)^k$, and $\left(1-\frac{4k}{n}\right)^{4k}$ are each at most $1 - \frac{16k^2}{n}$. After pulling out this factor, we obtain the inequality~\eqref{eq:embedding-average}.
\end{proof}

\section{Completing the proof of Theorem~\ref{main-theorem}}

To simplify notation, let $p_3^+(k) := \frac{4}{15} p_3(k) + \frac{6}{15} p_3(2k) + \frac{4}{15} p_3(3k) + \frac1{15} p_3(4k)$ (the quantity bounded by Lemma~\ref{lemma:k-embedding}) and let $q^+(k) := \frac{4}{15} q(k) + \frac{6}{15} q(2k) + \frac{4}{15} q(3k) + \frac1{15} q(4k)$. We noted previously that if $q(k) - \frac16 p_3(k) > \frac13$ for some $k$, then equation~\eqref{eq:ultimate} implies that $p_4(k)>0$, so a monochromatic line exists. Similarly, showing that $q^+(k) - \frac16 p_3^+(k) > \frac13$ is positive suffices: this is a weighted average, so $q(ik) - \frac16 p_3(ik) > \frac13$ will hold for some $1 \le i \le 4$.

We express as much of the left-hand side of Lemma~\ref{lemma:collinear-pairs} as possible in terms of $q^+$. We assume that the still-undetermined parameter $\kappa$ is a multiple of 4 for simplicity. In the sum
\begin{equation}
	\label{eq:q-plus-sum}
	\sum_{k=1}^{\kappa/4} (\kappa+1 - 2k) q^+(k)
\end{equation}
the coefficient of each $q(k)$ is $0$ for $k > \kappa$, and otherwise maximized if $k$ is divisible by both 3 and 4, in which case it is at most
\[
	\frac{4}{15}\left(\kappa+1 - 2\cdot k\right) + \frac{6}{15}\left(\kappa+1-2\cdot \frac{k}{2}\right) + \frac{4}{15}\left(\kappa+1- 2 \cdot \frac k3\right) + \frac1{15}\left(\kappa+1- 2 \cdot \frac k4\right) = \kappa + 1 - \frac{103}{90}k,
\]
so it is always less than $\kappa+1-k$. This means pulling out the sum~\eqref{eq:q-plus-sum} from the left-hand side of Lemma~\ref{lemma:collinear-pairs} leaves each $q(k)$ with a positive coefficient: we may write
\begin{equation}
	\label{eq:pull-out}
	\sum_{k=1}^\kappa (\kappa+1-k) q(k)  = \sum_{k=1}^{\kappa/4} (\kappa+1 - 2k) q^+(k) + \sum_{k=1}^\kappa R_k q(k)
\end{equation}
where $R_1, \dots, R_{\kappa}$ are all positive. Furthermore, though each $R_k$ is tedious to calculate, since equation~\eqref{eq:pull-out} is valid for all values of $q(1), \dots, q(\kappa)$, it remains valid if we set each of them to 1, and therefore
\[
	\sum_{k=1}^\kappa R_k = \sum_{k=1}^\kappa (\kappa+1-k) - \sum_{k=1}^{\kappa/4} (\kappa+1 - 2k) = \frac{\kappa^2+\kappa}{2} - \frac{3\kappa^2}{16} = \frac{5\kappa^2+8\kappa}{16}.
\]

If $q(k) > \frac12$ for any $k$, then from equation~\eqref{eq:density} we can conclude that $p_4(k) > 0$ and a monochromatic line exists. So assume the contrary: that $q(k) \le \frac12$ for all $k$. Then equation~\eqref{eq:pull-out} implies that
\[
	\sum_{k=1}^\kappa (\kappa+1-k) q(k)  \le \sum_{k=1}^{\kappa/4} (\kappa+1 - 2k) q^+(k) +  \frac{5\kappa^2+8\kappa}{16} \cdot \frac12.
\]
Therefore, by applying Lemma~\ref{lemma:collinear-pairs}, 
\[
	\sum_{k=1}^{\kappa/4} (\kappa+1 - 2k) q^+(k) \ge \frac{\kappa^2-1}{4} (1 - \epsilon(n, \kappa)) -  \frac{5\kappa^2+8\kappa}{32},
\]
where $\epsilon(n,\kappa)$ is the relative error term
\[
	\epsilon(n,\kappa) := e^{-(n-\kappa)/8} + 3 \kappa \sqrt{\frac{2}{\pi (n-\kappa)}}.
\]
Dividing by $\frac{3\kappa^2}{16}$ to obtain a weighted average and simplifying, we are left with
\[
	\frac{16}{3\kappa^2} \sum_{k=1}^{\kappa/4} (\kappa+1 - 2k) q^+(k) \ge \frac12 - \frac{4}{3\kappa} - \frac{4}{3}\epsilon(n,k) + \frac{4(1-\epsilon(n,k))}{3\kappa^2}.
\]
The last term is positive and may be dropped. Therefore there is some $k^* \le \kappa/4$ for which  $q^+(k^*) \ge \frac12 - \frac{4}{3\kappa} - \frac{4}{3}\epsilon(n,k)$.

On the other hand, Lemma~\ref{lemma:k-embedding} tells us that, as long as $4k^* < \sqrt{n}$, $p_3^+(k^*) \le \frac{14}{15} \left(1 - \frac{16 (k^*)^2}{n}\right)^{-1}$, which is at most $\frac{14}{15} \left(1 - \frac{\kappa^2}{n}\right)^{-1}$. Therefore a lower bound on $q^+(k^*) - \frac16 p_3^+(k^*) - \frac13$ is 
\begin{equation}
	\label{positive}
	\frac16 - \frac{4}{3\kappa} - \frac{4}{3}\epsilon(n,k) - \frac{7}{45}\left(1 - \frac{\kappa^2}{n}\right)^{-1},
\end{equation}
which is valid for any $\kappa < \sqrt{n}$.

As $n \to \infty$ and $\kappa \to \infty$, provided that $\frac{\kappa^2}{n} \to 0$, \eqref{positive} approaches $\frac1{90}$, and so a monochromatic line must exist. In particular, \eqref{positive} is already positive for $n=10^{11}$ and $\kappa=368$, completing the proof. (More precisely, $n=19\,012\,590\,257$ and $\kappa=240$ are enough.)

\bibliographystyle{plain}

\begin{thebibliography}{1}

\bibitem{alon08}
N.~Alon and J.~Spencer.
\newblock {\em The Probabilistic Method}.
\newblock John Wiley \& Sons, Hoboken, NJ, 2008.

\bibitem{berlekamp68}
E.~R. Berlekamp.
\newblock A construction for partitions which avoid long arithmetic
  progressions.
\newblock {\em Canad. Math. Bull.}, 11:409--414, 1968.

\bibitem{chvatal70}
V.~Chv{\'a}tal.
\newblock Some unknown van der {W}aerden numbers.
\newblock In {\em Combinatorial {S}tructures and their {A}pplications ({P}roc.
  {C}algary {I}nternat. {C}onf., {C}algary, {A}lta., 1969)}, pages 31--33.
  Gordon and Breach, New York, 1970.

\bibitem{hales63}
A.~W. Hales and R.~I. Jewett.
\newblock Regularity and positional games.
\newblock {\em Trans. Amer. Math. Soc.}, 106:222--229, 1963.

\bibitem{hindman14}
Neil Hindman and Eric Tressler.
\newblock The first nontrivial {H}ales-{J}ewett number is four.
\newblock {\em Ars Combin.}, 113:385--390, 2014.

\bibitem{shelah88}
Saharon Shelah.
\newblock Primitive recursive bounds for van der {W}aerden numbers.
\newblock {\em J. Amer. Math. Soc.}, 1(3):683--697, 1988.

\end{thebibliography}

\end{document}